\documentclass[12pt,a4paper,oneside,intlimits,sumlimits,reqno]{amsart}
\pdfoutput=1
\usepackage{color,amssymb,latexsym,amsfonts,textcomp,fancyhdr,calc,graphicx}
\usepackage{amsmath,amstext,amsthm,amssymb,amsxtra}
\usepackage[left=2.5cm,right=2.5cm,bottom=2cm,top=2cm]{geometry} %%%%%%%
\usepackage[colorlinks=true]{hyperref}
\usepackage{asymptote}
\usepackage[integrals]{wasysym}

\usepackage{txfonts,pxfonts,tikz} %also txfonts 
\usepackage[T1]{fontenc}

\usepackage[
    backend=biber,
    style=numeric,
    natbib=true,
    url=false, 
    doi=true,
    eprint=false
]{biblatex}
\newtheorem{theorem}{Theorem}

\newtheorem{lemma}{Lemma}

\newtheorem{definition}{Definition}[section]
\theoremstyle{definition}
\newtheorem{example}{Example}

\newcommand{\beql}[1]{\begin{equation}\label{#1}}
\newcommand{\eeq}{\end{equation}}
\newcommand{\comment}[1]{}

\newcommand{\Abs}[1]{{\left|{#1}\right|}}

\newcommand{\Set}[1]{{\left\{{#1}\right\}}}

\newcommand{\RR}{{\mathbb R}}

\newcommand{\CC}{{\mathbb C}}
\newcommand{\ZZ}{{\mathbb Z}}
\newcommand{\NN}{{\mathbb N}}
\newcommand{\TT}{{\mathbb T}}
\newcommand{\QQ}{{\mathbb Q}}

\newcommand{\one}{{\bf 1}}

\newcommand{\supp}{{\rm supp\,}}

\newcommand{\ft}[1]{\widehat{#1}}

\newcounter{rem}
\newcounter{step}
\newcounter{mysec}
\newcounter{mysubsec}[mysec]
\newcounter{othm}
\def\theothm{\Alph{othm}}

\begin{document}

\title{The structure of multiplicative tilings of the real line}
\author{Mihail N. Kolountzakis and Yang Wang}
\address{M.K.: Department of Mathematics and Applied Mathematics, University of Crete, Voutes Campus, GR-700 13, Heraklion, Crete, Greece}
\email{kolount@gmail.com}
\address{Y.W.: Department of Mathematics, Hong Kong Univ.\ of Science and Technology, Kowloon, Hong Kong}
\email{yangwang@ust.hk}
\thanks{
The first author has been partially supported by the Hong Kong Univ.\ of Science and Technology. The second author is supported in part by the Hong Kong Research Grant Council grant 16306415 and 16317416 and by University of Crete grant No 4725. 
}

\maketitle
\tableofcontents

\begin{abstract}
Suppose $\Omega, A \subseteq \RR\setminus\Set{0}$ are two sets, both of mixed sign, that $\Omega$ is Lebesgue measurable
and $A$ is a discrete set. We study the problem of when $A \cdot \Omega$ is a (multiplicative) tiling of the real line,
that is when almost every real number can be uniquely written as a product $a\cdot \omega$, with $a \in A$, $\omega \in \Omega$.
We study both the structure of the set of multiples $A$ and the structure of the tile $\Omega$. We prove strong results in both cases.
These results are somewhat analogous to the known results about the structure of translational tiling of the real line.
There is, however, an extra layer of complexity due to the presence of sign in the sets $A$ and $\Omega$,
which makes multiplicative tiling roughly equivalent to translational tiling on the larger group $\ZZ_2 \times \RR$.

\end{abstract}

%%%%%%%%%%%%%%%%%%%%%%%%%%%%%%%%%%%%%%%%%%%%%%%%%%%%%%%%%%%%%%%%%%%%%%%%%%%%%%%%%
%%%%%%%%%%%%%%%%%%%%%%%%%%%%%%%%%%%%%%%%%%%%%%%%%%%%%%%%%%%%%%%%%%%%%%%%%%%%%%%%%
%%%%%%%%%%%%%%%%%%%%%%%%%%%%%%%%%%%%%%%%%%%%%%%%%%%%%%%%%%%%%%%%%%%%%%%%%%%%%%%%%
\section{Introduction}
\label{sec:intro}

Tilings have long fascinated mathematicians \cite{grunbaum1986tilings}.
The case where one moves a single object by translation in an abelian group (translational tiling)
has proved both challenging and full of connections to Functional Analysis \cite{kolountzakis2004milano}, 
such as connections to the so-called Fuglede Conjecture or Spectral Set Conjecture \cite{fuglede1974operators,tao2004fuglede,kolountzakis2006tiles}.
Simultaneous tiling of a set by both translation and multiplication (with linear operators on the space where tiling takes place)
has also been studied mainly because of its connections to wavelets
\cite{wang2002wavelets,speegle2003dilation,olafsson2004wavelets,dobrescu2006wavelet,ionascu2006simultaneous}.

\begin{definition}[Translational tiling, multiplicative tiling]\label{def:tiling}\ \\
Suppose $f:\RR^d\to\CC$ is measurable and $A \subseteq \RR^d$. We say that $f+A$ is a tiling
of $\RR^d$ at level $\ell$ if
$$
\sum_{a \in A} f(x-a) = \ell
$$
for almost all $x \in \RR^d$ with the sum converging absolutely.

If $\Omega \subseteq \RR^d$ is a measurable set and $f = \one_\Omega$ we also say that $\Omega+A$ is a tiling.

If $A \subseteq \RR\setminus\Set{0}$ and
$$
\sum_{a \in A} f(a^{-1} x) = \ell
$$
for almost all $x \in \RR^d$, with absolute convergence, then we say that $A\cdot f$ is a tiling.

If $\Omega \subseteq \RR^d$ is a measurable set and $f = \one_\Omega$ we also say that $A\cdot \Omega$ is a tiling.

\end{definition}

While translational tiling or more generally tiling using congruent tiles has been studied extensively, one particular tiling, namely {\em multiplicative tiling}, has not. Such tiling arise rather ubiquitously in the study wavelets and wavelet sets. In the standard setting, a function $f(x) \in L^2(\RR)$ is a {\em wavelet} if $\{2^{j/2}f(2^jx-k): ~j,k\in\ZZ\}$ form an orthonormal basis for $L^2(\RR)$. A set $\Omega\subset \RR$ is a {\em wavelet set} if $\ft\chi_\Omega$ is a wavelet. It was first shown by Dai and Larson \cite{dai1998wandering} that $\Omega$ is a wavelet set if and only if it tiles $\RR$ translationally by $\ZZ$ and multiplicatively by the set $\{2^j:~j\in\ZZ\}$, see also \cite{dai1997wavelet,speegle2003dilation}. The more general multiplicative tiling, which we aim to study here, was first introduced in Wang \cite{wang2002wavelets} to study a more general form of wavelet sets.

Our purpose here is to study the structure of multiplicative tilings.
We are guided in this by previous work on the structure of translational tilings of the real line or of the integer line.
In \cites{newman1977tesselations,leptin1991uniform,lagarias1996tiling,kolountzakis1996structure}
it is proved, under very broad conditions, that any translational tiling of the real line (or of the integer line) must be periodic.
The main tool in the study of translational tilings by a single tile has long been (see e.g.\ \cite{kolountzakis2004milano}) Fourier Analysis,
where the problem is expressed roughly as a support condition of the Fourier Transform of the set of translates
on the zero set of the Fourier Transform of the tile, an approach that will also be used extensively in this paper.

Suppose then that $A \subseteq \RR\setminus\Set{0}$ is a discrete set and $\Omega \subseteq \RR$ is a measurable set.
We want to derive properties of $\Omega$ and $A$ from the assumption of multiplicative tiling $A\cdot \Omega = \RR$ (multiplicative tiling at level 1).

\noindent{\bf The importance of sign.}

It is important to emphasize that if $A$ or $\Omega$ are nonnegative (or of one sign, more generally) then the problem
quickly reduces to translational tiling.
Indeed, suppose that $\Omega \subseteq (0,+\infty)$.
Then, writing $A = A^+ \cup (-A^-)$, with $A^\pm \subseteq (0,+\infty)$, we see that
the tiling condition $A\cdot\Omega = \RR$ is equivalent to the two tiling conditions
$$
\RR^+ = A^+\cdot \Omega\ \ \ \mbox{and}\ \ \ \RR^+ = A^-\cdot\Omega.
$$
Taking logarithms of both we reduce $A\cdot\Omega=\RR$ to the two independent translational tilings
$$
\RR = \log\Omega + \log A^+\ \ \ \mbox{and}\ \ \ \RR = \log\Omega + \log A^-.
$$
So if one can understand translational tiling by the set $\log\Omega$ then all results transfer back to our multiplicative
tiling $A\cdot\Omega = \RR$ if $\Omega$ is of one sign.

Similarly, if $A \subseteq (0, +\infty)$ then, writing again $\Omega = \Omega^+ \cup (-\Omega^-)$, with $\Omega^\pm \subseteq (0, +\infty)$,
we obtain that $A\cdot \Omega = \RR$ is equivalent to the two translational tilings
$$
\RR = \log\Omega^+ + \log A \ \ \ \mbox{and}\ \ \ \RR = \log\Omega^- + \log A.
$$

If however the two sets $A$ and $\Omega$ have both a positive and a negative part the multiplicative tiling $A\cdot \Omega = \RR$
becomes a lot more complicated. It still reduces to tiling by translation but not of the ordinary kind with one set being
translated around to fill space.
Indeed, when $A = A^+ \cup (-A^-)$ and $\Omega = \Omega^+ \cup (-\Omega^-)$ then the multiplicative tiling $A \cdot \Omega = \RR$
reduces to the two simultaneous tilings
$$
\RR^+ = A^+\Omega^+ \cup A^-\Omega^-\ \ \ \mbox{and}\ \ \ \RR^+ = A^-\Omega^+ \cup A^+\Omega^-,
$$
and, after taking logarithms, to the sumultaneous translational tiling
\begin{align}
\RR &= (\log\Omega^++\log A^+) \cup (\log\Omega^-+\log A^-)\ \ \ \mbox{and}\ \ \ \label{cti}\\
\RR &= (\log\Omega^++\log A^-) \cup (\log\Omega^-+\log A^+).\nonumber
\end{align}
The meaning of the notation here should be obvious. For instance, the meaning of the first equation in \eqref{cti} is that
almost every point in $\RR$ can be written, uniquely, either in the form $\log\omega+\log a$, with $\omega\in\Omega^+, a \in A^+$,
or in the form $\log\omega+\log a$, with $\omega \in \Omega^-, a \in A^-$.
Put differently, the translates of the set $\log\Omega^+$ by the numbers in $\log A^+$ together with the translates of the set $\log\Omega^-$
by the numbers in $\log A^-$ cover almost all of $\RR$ exactly once and any two of these sets intersect at a set of measure zero.

The purpose of this paper is first, to exploit \eqref{cti} in order to derive structural properties of the
set $\Omega$ (the tile) and the set $A$ (the set of multiples) and, second, to study \eqref{cti} (which we call {\em cross-tiling})
in itself, and in the case of a finite cyclic group, where things are simpler. In particular, we want to make some
connections and distinctions to ordinary translational tiling where only one set is translated.

The structure of the paper is as follows.
In \S\ref{sec:rationality} we restrict ourselves to translational tilings and generalize some periodicity and rationality results
(from \cite{kolountzakis1996structure,lagarias1996tiling}) to the extent that they become useful to us in the analysis of \S\ref{sec:structure}
and \S\ref{sec:tile-structure}
where structure results are proved for the logarithms of the sets $A$ and $\Omega$ respectively.
In \S\ref{sec:ct} the problem of cross tiling is studied in cyclic groups (we show in \S\ref{sec:structure} that multiplicative tiling of $\RR$
reduces to cross tiling in cyclic groups), and we provide examples of cross tilings which differ significantly from ordinary translational
tilings as well as a Fourier condition for cross tiling, analogous to the one for ordinary translational tilings.

%%%%%%%%%%%%%%%%%%%%%%%%%%%%%%%%%%%%%%%%%%%%%%%%%%%%%%%%%%%%%%%%%%%%%%%%%%%%%%%%%
%%%%%%%%%%%%%%%%%%%%%%%%%%%%%%%%%%%%%%%%%%%%%%%%%%%%%%%%%%%%%%%%%%%%%%%%%%%%%%%%%
%%%%%%%%%%%%%%%%%%%%%%%%%%%%%%%%%%%%%%%%%%%%%%%%%%%%%%%%%%%%%%%%%%%%%%%%%%%%%%%%%
\section{The structure of multiple translational tiling by a set}
\label{sec:rationality}

%%%%%%%%%%%%%%%%%%%%%%%%%%%%%%%%%%%%%%%%%%%%%%%%%%%%%%%%%%%%%%%%%%%%%%%%%%%%%%%%%
\begin{lemma}\label{lm:poly-zeros}
Suppose $\Lambda$ is a finite subset of the torus $\TT=\RR/\ZZ$ and
\beql{expoly}
f(n) = \sum_{\lambda\in\Lambda} c_\lambda e^{2\pi i \lambda n},\ \ \ (n\in\ZZ)
\eeq
is an exponential polynomial on the integers ($c_\lambda \in \CC$).
Suppose that
$$
\Lambda = \Lambda_1 \cup \cdots \cup \Lambda_r,\ \ \ (r\ge 1)
$$
is the decomposition of $\Lambda$ into rational equivalence classes (two points in $\Lambda$ have rational difference
if and only if they belong to the same $\Lambda_j$).
Write also $f_j(n) = \sum_{\lambda\in\Lambda_j} c_\lambda e^{2\pi i \lambda n}$ so that $f=f_1+\cdots+f_r$.

Then the zeros of $f$ are the common zeros of the $f_j$ plus a finite set (possibly empty).
\end{lemma}

\begin{proof}
Write $Z(\phi)$ for the zero set of a function $\phi$ on its domain.

Define the set of integers
$$
X = Z(f) \setminus \bigcap_{j=1}^r Z(f_j).
$$
We have to show that $X$ is finite.
By the Skolem-Mahler-Lech Theorem \cite{lech1953note} the integer zero set of every exponential polynomial, such as \eqref{expoly},
is a periodic set plus a finite set (possibly empty). Since $\Abs{f_1}^2+\cdots+\Abs{f_r}^2$ is also an exponential
polynomial it follows that both $Z(f)$ and $\bigcap_{j=1}^r Z(f_j)$ are periodic sets plus a finite set.
Therefore $X$ is also a periodic set, give or take a finite set.

It suffices therefore to prove that $X$ does not contain arithmetic progressions of arbitrary length, as then
it follows that $X$ has no periodic part and is just a finite set.

For $j=1,2,\ldots,r$
write $\Lambda_j = a_j + Q_j$, where $Q_j \subseteq \QQ$ is a finite set and $a_i-a_j \notin \QQ$ for $i \neq j$.
Let $N \in \NN$ be the least common multiple of all denominators in all the $Q_j$, so that $Nq \in \ZZ$ for all $q \in \bigcup_{j=1}^r Q_j$.
If $X$ contains arbitrarily long arithmetic progressions then it contains a progression of the form
\beql{assumption1}
a+bNk,\ \ \ (k=0, 1, 2,\ldots, r)
\eeq
for some $a, b \in \ZZ$, $b>0$.

For each $k=1,2,\ldots,r$ we then have
\begin{align}
f(a+bNk) &= \sum_{j=1}^r \sum_{q\in Q_j} c_{a_j+q} e^{2\pi i (a_j+q) (a+bNk)} \nonumber \\
 &= \sum_{j=1}^r z_j^k x_j \label{vandermonde}
\end{align}
with
\begin{align*}
z_j &= e^{2\pi i a_j b N} \\
x_j &= e^{2\pi i a_j a} \sum_{q \in Q_j} c_{a_j+q} e^{2\pi i qa} = f_j(a).
\end{align*}
All numbers $z_j$ are different since the differences of the $a_j$'s are irrational, so the Vandermonde linear system
$$
\sum_{j=1}^r z_j^k x_j = 0,\ \ \ (j=1,2,\ldots,r)
$$
which we obtain if we assume that $a+bNk \in X$, for $k=1,2,\ldots,r$, has only the all-zero solution $x_1=\cdots=x_r=0$, which implies
that $f_1(a)=\cdots=f_r(a)=0$, hence $a$ is a common zero of all $f_j$, hence not in $X$, a contradiction with \eqref{assumption1} for $k=0$.
So $X$ does not contain arbitrarily long arithmetic progressions and is, therefore, a finite set.
\end{proof}

%%%%%%%%%%%%%%%%%%%%%%%%%%%%%%%%%%%%%%%%%%%%%%%%%%%%%%%%%%%%%%%%%%%%%%%%%%%%%%%%%
\begin{lemma}\label{lm:torus-tiling}
Suppose $\Lambda$ is a finite subset of the torus $\TT=\RR/\ZZ$ and
$$
\Lambda = \Lambda_1 \cup \cdots \cup \Lambda_r,
$$
is the decomposition of $\Lambda$ into rational equivalence classes.
Suppose also that $F \in L^1(\TT)$ and $c_\lambda \in \CC$ are such that
\beql{torus-tiling}
\sum_{\lambda\in\Lambda} c_\lambda F(x-\lambda) = \text{const.\ \ for almost all $x \in \TT$}.
\eeq
If the function $F$ takes only countably many values then for each $j=1,2,\ldots,r$ we also have
$$
\sum_{\lambda \in \Lambda_j} c_\lambda F(x-\lambda) = \text{const.${}_j$\ \ for almost all $x \in \TT$}.
$$
\end{lemma}

\begin{proof}
Our assumption \eqref{torus-tiling} is equivalent to
$$
\ft{F}(n) = 0\ \ \text{or}\ \ \sum_{\lambda\in\Lambda}c_\lambda e^{2\pi i \lambda n} = 0,\ \ \ (n \neq 0).
$$
In other words we must have
$$
Z(\ft{F}) \cup Z\left(\sum_{\lambda\in\Lambda}c_\lambda e^{2\pi i \lambda n}\right) \supseteq \ZZ\setminus\Set{0}.
$$
But, from Lemma \ref{lm:poly-zeros},
$$
Z\left(\sum_{\lambda\in\Lambda}c_\lambda e^{2\pi i \lambda n}\right) \setminus \bigcap_{j=1}^r\ Z\left(\sum_{\lambda\in\Lambda_j}c_\lambda e^{2\pi i \lambda n}\right)
$$
is a finite set.
This implies that
$Z(\ft{F}) \cup Z(\sum_{\lambda\in\Lambda_j}c_\lambda e^{2\pi i \lambda n})$ contains all but finitely many integers, for each $j=1,2,\ldots,r$.
Consequently the function
\beql{torus-tiling-j}
\sum_{\lambda \in \Lambda_j} c_\lambda F(x-\lambda)
\eeq
is a trigonometric polynomial of $x$. But, as $F$ takes only countably many values and this is a finite sum, the function
in \eqref{torus-tiling-j} has a countable range too, and this can only happen
if the function is a constant, which is exactly what we wanted to prove.
\end{proof}

%%%%%%%%%%%%%%%%%%%%%%%%%%%%%%%%%%%%%%%%%%%%%%%%%%%%%%%%%%%%%%%%%%%%%%%%%%%%%%%%%
\begin{lemma}\label{lm:integer-tiling}
Suppose that $V \subseteq \ZZ\setminus\Set{0}$ is a finite set of non-zero integers, $A \subseteq \RR$ is a
discrete set of bounded density and $v_t \in V$, for $t \in A$, are such that
\beql{int-tiling}
\sum_{t \in A} v_t f(x-t) = k,\ \ \ \mbox{for almost all $x \in \RR$},
\eeq
where $f \in L^1(\RR)$ is an integer-valued function of compact support and $k$ is an integer.
Then the measure
$$
\mu = \sum_{t \in A} v_t \delta_t
$$
is a periodic measure (and, therefore, $A$ is a periodic set).
\end{lemma}

\begin{proof}
It follows from the proof of \cites[Theorem 3.1]{kolountzakis1996structure} that
\beql{mu-ft}
\supp\ft{\mu} \subseteq \Set{0} \cup \Set{\xi: \ft{f}(\xi)=0},
\eeq
with the right-hand side of the above equation being a discrete set (since $f$ has compact support $\ft{f}$ is analytic).
For $v \in V$ write $A^v = \Set{t \in A: v_t = v}$.
It follows from the proof of \cites[Theorem 5.1]{kolountzakis1996structure}
that each $A^v$ has the form
\beql{av-qp-1}
A^v = F^v \triangle \bigcup_{j=1}^{J^v} \left( \alpha_j^v \ZZ + \beta_j^v \right),
\eeq
for some $J^v \in \NN$, $\alpha_j^v > 0, \beta_j^v \in \RR$ and finite set $F^v \subseteq \RR$.
(Let us only indicate that, as in \cites{kolountzakis1996structure}, the main ingredient
in the proof of \eqref{av-qp-1} is a theorem of Meyer \cites[Theorem 4.2]{kolountzakis1996structure}, \cites{meyer1970nombres},
itself a consequence of the Idempotent Theorem of P.J. Cohen \cites[Theorem 4.1]{kolountzakis1996structure}, \cites{cohen1959homomorphisms}.)

By merging together the $a_j^v$ which are commensurable we can rewrite \eqref{av-qp-1} as
\beql{av-qp}
\delta_{A^v} = \sum_{i=1}^{I^v} \delta_{\gamma_i^v\ZZ}*\nu_i^v + \nu^v,
\eeq
where all the $\gamma_i^v$ have irrational ratios and the measures $\nu_i^v$ and $\nu^v$
are each a finite sum of integer point masses on $\RR$.

Using \eqref{av-qp} we can now write
$$
\mu = \sum_{v\in V} \delta_{A^v} = \sum_{v \in V} v \sum_{i=1}^{I^v} \delta_{\gamma_i^v\ZZ}*\nu_i^v + \sum_{v \in V} v \nu^v.
$$
Taking Fourier Transforms above we observe that the first summand on the right contributes a discrete measure to $\ft{\mu}$
(by the Poisson Summation Formula) and the second summand contributes a trigonometric polynomial. But since, by \eqref{mu-ft},
the Fourier Transform of $\mu$ must have a discrete support it follows that the second summand is 0 and we have
\beql{mu-decomp}
\mu = \sum_{k=1}^K \delta_{\zeta_k \ZZ}*\tau_k,
\eeq
where (having, again, merged the arithmetic progressions with commensurable periods) the ratio of any two $\zeta_k$ is irrational
and the $\tau_k$ are finite sums of integer point masses on $\RR$.
Taking Fourier Transforms we get by the Poisson Summation Formula that
$$
\ft{\mu}(\xi) = \sum_{k=1}^K \ft{\tau_k}(\xi) \zeta_k^{-1} \delta_{\zeta_k^{-1}\ZZ},
$$
and observe that the measures $\ft{\tau_k}(\xi) \zeta_k^{-1} \delta_{\zeta_k^{-1}\ZZ}$ have disjoint supports except at $0$.
Using our assumption \eqref{int-tiling} that $f*\mu = k$ we obtain the tilings
$$
f*(\delta_{\zeta_k \ZZ}*\tau_k) = C_k,
$$
where $C_k$ is also an integer constant.
Integrating this over one period $[0,\zeta_k)$ we obtain
$$
C_k \zeta_k = \int f \cdot \tau_k([0,\zeta_k)).
$$
This shows that all $\zeta_k$ are rational multiples of $\int f$
so all summands in \eqref{mu-decomp} can be merged to one
$$
\mu = \delta_{\zeta\ZZ}*\tau,
$$
where $\tau$ is, again, a finite sum of integer point masses, hence $\mu$ is a periodic measure with period $\zeta$, as
we had to prove.
\end{proof}

%%%%%%%%%%%%%%%%%%%%%%%%%%%%%%%%%%%%%%%%%%%%%%%%%%%%%%%%%%%%%%%%%%%%%%%%%%%%%%%%%
%%%%%%%%%%%%%%%%%%%%%%%%%%%%%%%%%%%%%%%%%%%%%%%%%%%%%%%%%%%%%%%%%%%%%%%%%%%%%%%%%
\begin{theorem}\label{th:integer-tiling}
Suppose that $V \subseteq \ZZ$ is a finite set of non-zero integers, $A \subseteq \RR$ is a
discrete set of bounded density and $v_t \in V$, for $t \in A$, are such that
\beql{integer-tiling}
\sum_{t \in A} v_t f(x-t) = k,\ \ \ \mbox{for almost all $x \in \RR$},
\eeq
where $f \in L^1(\RR)$ is an integer-valued function of compact support and $k$ is an integer.
Then
\begin{enumerate}
\item[(i)]\label{en:periodic}
The measure $\mu = \sum_{t\in A} v_t \delta_t$ is periodic and can be written in the form
$$
\mu = \delta_{\zeta\ZZ}*\tau,
$$
where $\zeta>0$ and $\tau$ is a finite sum of integer point masses
$$
\tau = \sum_{s=1}^S c_s \delta_{x_s},
$$
where $c_s \in \ZZ$, $x_s \in [0,\zeta)$.
\item[(ii)]\label{en:split}
Write $X = \Set{x_1, x_2, \ldots, x_S}$ and
$$
X = X_1 \cup \cdots \cup X_r
$$
for the partition of $X$ into equivalence classes mod $\zeta\QQ$.
Then for $j=1,2,\ldots,r$ and with $\tau_j = \sum_{x \in X_j} c_x \delta_x$ we have
the tilings
$$
f*\delta_{\zeta\ZZ}*\tau_j = k_j,
$$
for some integers $k_j$, $j=1,2,\ldots,r$.
\end{enumerate}
\end{theorem}

Theorem \ref{th:integer-tiling} was proved in \cites{lagarias1996tiling} for
$f$ being the indicator function of a bounded, measurable subset of $\RR$,
and with $v_t = 1$ for all $t \in A$.
In this case the number $r$ of classes in \eqref{en:split} is 1, and the tiling set $A$ is therefore rational, i.e.\ the differences of its
elements are rational multiples of the period.
The proof does not readily extend to the more general case of Theorem \ref{th:integer-tiling} and this is what we show here.

\begin{proof}
Part \eqref{en:periodic} of the Theorem is merely a restatement of Lemma \ref{lm:integer-tiling}.

Notice that we can assume from now on that $\zeta$ (the period of the tiling) is 1, as we can dilate the axis.

Define the $\ZZ$-periodization of $f$
$$
F(x) = \sum_{n \in \ZZ} f(x-n) = f*\delta_\ZZ(x),
$$
which is in $L^1(\TT)$,
and observe that the tiling $f*(\delta_\ZZ*\tau)(x)=k$ is equivalent to the tiling of $\TT$
$$
F*\tau(x)=k,\ \ \ \mbox{for almost all $x \in \TT$}.
$$
Since $F$ is also integer-valued Lemma \ref{lm:torus-tiling} applies and we conclude that
$$
F*\tau_j(x)=k_j,\ \ \ \mbox{for almost all $x \in \TT$ and some integer $k_j$},
$$
which is equivalent to $f*(\delta_\ZZ*\tau_j)=k_j$ as we had to prove.
This concludes the proof of \eqref{en:split}.
\end{proof}

%%%%%%%%%%%%%%%%%%%%%%%%%%%%%%%%%%%%%%%%%%%%%%%%%%%%%%%%%%%%%%%%%%%%%%%%%%%%%%%%%
%%%%%%%%%%%%%%%%%%%%%%%%%%%%%%%%%%%%%%%%%%%%%%%%%%%%%%%%%%%%%%%%%%%%%%%%%%%%%%%%%
%%%%%%%%%%%%%%%%%%%%%%%%%%%%%%%%%%%%%%%%%%%%%%%%%%%%%%%%%%%%%%%%%%%%%%%%%%%%%%%%%
\section{The structure of the set of multiples of a multiplicative tile}
\label{sec:structure}

\begin{theorem}\label{th:mult-structure}
(Structure of the set of multiples)\\
Suppose $\Omega \subseteq \RR$ is a bounded measurable set such that $\Omega \cap (-\epsilon, \epsilon) = \emptyset$ for some $\epsilon>0$.
Suppose also $A \subseteq \RR\setminus\Set{0}$ is a discrete set such that
$$
A \cdot \Omega
$$
is a (multiplicative) tiling of $\RR$ at level 1.

Let $\Omega^+, \Omega^-, A^+, A^- \subseteq \RR^+$ be the positive and negative parts of $\Omega$ and $A$
$$
\Omega^+ = \Omega \cap (0, +\infty),\ \Omega^- = -(\Omega \cap (-\infty, 0)),\ 
A^+ = A \cap (0, +\infty),\ A^- = -(A \cap (-\infty, 0)).
$$
Then
\begin{enumerate}
\item[(i)]\label{symmetric-case} If $\Omega$ is \underline{essentially symmetric} (i.e.\ if $\Abs{\Omega^+ \, \triangle\, \Omega^-} = 0$) then
$A^+ \cap A^- = \emptyset$ and the union $\log A^+ \cup \log A^-$ is periodic of the form
\beql{periodic-rational}
\alpha\ZZ+\Set{r_1, r_2, \ldots, r_s}\ \ \ \text{with $r_i-r_j$ rational multiples of $\alpha>0$}.
\eeq
The partition of the set $\log A^+ \cup \log A^-$ into its component sets $\log A^+$ and $\log A^-$
can be completely arbitrary.
\item[(ii)]\label{nonsymmetric-case} If $\Omega$ is \underline{not essentially symmetric} with respect to the origin then 
the sets $\log A^+$ and $\log A^-$ are both periodic and of the form \eqref{periodic-rational} with the same period $\alpha$.
\end{enumerate}
\end{theorem}

\begin{proof}
In order to transfer the problem to the translational case, which is much better understood, it is natural to take logarithms.
Allowing ourselves a slight abuse of terminology,
we then have that $A\cdot \Omega=\RR$ is a tiling if and only if
$$
\RR^+ = A^+ \Omega^+ \cup A^- \Omega^-,\ \ \ \RR^+ = A^- \Omega^+ \cup A^+ \Omega^-
$$
are both tilings (where $\RR^+$ is the right half-line).

Taking logarithms of both we obtain the additive (translational) tilings
\beql{log-tilings}
\RR = (a^+ + \omega^+) \cup (a^- + \omega^-) = (a^-+\omega^+) \cup (a^+ + \omega^-),
\eeq
where we write the lower case letter for the set of logarithms of a set written with the corresponding capital letter,
e.g.\ $a^+ = \log A^+$.
Identifying, further, the sets $\omega^\pm$ with their indicator functions
and the discrete sets $a^{\pm}$ with a collection of unit point masses at their points (for instance, we write $a^+$ instead of $\delta_{a^+}$),
we may rewrite the above tilings using convolution as
\beql{tilings}
1 = a^+*\omega^+(x) + a^-*\omega^-(x) = a^-*\omega^+(x) + a^+*\omega^-(x),\ \ \ \mbox{for almost all $x \in \RR$}.
\eeq
Adding and subtracting the two identities in \eqref{tilings} we get the equivalent set of identities (valid for almost all $x \in \RR$)
\beql{sum}
2 = (\omega^++\omega^-)*(a^++a^-)(x)
\eeq
and
\beql{diff}
0 = (\omega^+ - \omega^-)*(a^+-a^-)(x).
\eeq
Notice that $\omega^+(x)+\omega^-(x)$ is a function that only takes the values 0, 1 and, possibly, 2 and that $a^++a^-$ is a measure,
which is a collection of point masses of weight 1 or 2.
Similarly $a^+-a^-$ is a measure which is a collection of point masses of weight $\pm 1$.

There is obviously no problem with the definition of the convolutions in \eqref{tilings} and \eqref{sum}, as there are
only nonnegative quantities involved. A moment's thought should convince us that there is no problem in \eqref{diff} either,
as all sums involved have finitely many terms,
the functions $\omega^\pm(x)$ being of compact support and the sets $a^\pm$ being discrete.

{\bf Periodicity.}
From Theorem \ref{th:integer-tiling}\eqref{en:periodic} applied to the tiling \eqref{sum} we obtain that
$a$ (viewed as a multiset when the point mass at a point has weight 2)
is a periodic set
\beql{periodic-set}
a = a^+ \cup a^- =  \gamma\ZZ + \Set{\beta_1, \beta_2, \ldots, \beta_J},
\eeq
for some $\gamma>0$, $\beta_j \in \RR$.
As a consequence the set
$$
a^+ \cap a^-
$$
is also periodic, as this is where the multiplicity of $a$ is equal to 2.

Assume now that $\omega^+$ is not identical to $\omega^-$ (that is $\Omega$ is not symmetric with respect to 0)
so that the function $\omega^+-\omega^-$ that appears in \eqref{diff} is not equal to zero almost everywhere.
The measure $a^+-a^-$ is a collection of Dirac point masses of weight $\pm 1$. The weight 1 appears exactly on
the points of the set $a^+ \setminus a^-$ and the weight -1 exactly on the set $a^- \setminus a^+$.
Given that we have already established the periodicty of $a^+ \cap a^-$
the periodicity of the set $a^+$ and the periodicity of the set $a^-$ will follow if we manage to show
the periodicity of $a^+ \setminus a^-$ and of $a^- \setminus a^+$ with a period commensurable to a period of $a^+ \cap a^-$.

It follows again from Theorem \ref{th:integer-tiling}\eqref{en:periodic}, applied to tiling \eqref{diff},
with $f = \omega^+-\omega^-$ (this is a compactly supported function, because of our assumption
that the bounded set $\Omega$ avoids an open neighborhood of $0$),
that the sets $a^+\setminus a^-$ and $a^-\setminus a^+$ are periodic with the same period.
Since we have already shown that $a^+ \cap a^-$ is also periodic it follows that each of the sets $a^+$ and $a^-$
is a union of two periodic sets, and these must be of commensurable periods, as, otherwise, the set
$a^+ \cup a^-$, which is already known to be periodic, would contain two arithmetic progressions with inocommensurable step,
an impossibility.
It follows that $a^+$ and $a^-$ are periodic too, and with commensurable periods.

{\bf Rationality.}
Dilating space we may assume now that the sets $a^\pm$ have period 1.
It remains to prove that the set $a^+ \cup a^-$ has rational differences.
Write
\beql{one-period-plus}
a^+ + a^- = \delta_\ZZ*\sigma,\ \ \ \mbox{with}\ \ \ \sigma = \sum_{j=1}^J n_j \delta_{x_j},
\eeq
where $n_j \in \Set{1,2}$ and $x_j \in [0,1)$, for $j=1, 2, \ldots, J$.
It follows that we can also write
\beql{one-period-minus}
a^+-a^- = \delta_\ZZ*\tau,\ \ \ \mbox{with}\ \ \ \tau = \sum_{j=1}^J m_j \delta_{x_j},
\eeq
where now $m_j \in \Set{-1, 0, 1}$, for $j=1, 2, \ldots, J$.

Contrary to what we want to prove, assume that not all the $x_j$s are rationally equivalent, and
let $x_{i_1}, x_{i_2}, \ldots, x_{i_r}$ be a rationally equivalent class of the points $x_1, x_2, \ldots, x_J$ in \eqref{one-period-plus}.
We now apply Theorem \ref{th:integer-tiling}\eqref{en:split} to the two tilings \eqref{sum} and \eqref{diff}.
It follows that we have the two tilings
\beql{sum-1}
(\omega^++\omega^-)*(\delta_\ZZ+\sum_{j=1}^r n_{i_j} \delta_{x_{i_j}}) = k_1
\eeq
and
\beql{diff-1}
(\omega^+-\omega^-)*(\delta_\ZZ+\sum_{j=1}^r m_{i_j} \delta_{x_{i_j}}) = k_2,
\eeq
for some two integers $k_1, k_2$.

Write $a^{'\pm}$ for the restriction of $a^\pm$ to the set $\ZZ+\Set{x_{i_1}, \ldots, x_{i_r}}$
and observe that we now have
\beql{sum-2}
a^{'+}+a^{'-} = \delta_\ZZ + \sum_{j=1}^r n_{i_j} \delta_{x_{i_j}}
\eeq
and
\beql{diff-2}
a^{'+}-a^{'-} = \delta_\ZZ + \sum_{j=1}^r m_{i_j} \delta_{x_{i_j}},
\eeq
so that the tilings \eqref{sum-1} and \eqref{diff-1} can now be written as
\beql{sum-3}
(\omega^++\omega^-)*(a^{'+}+a^{'-}) = k_1
\eeq
and
\beql{diff-3}
(\omega^+-\omega^-)*(a^{'+}-a^{'-}) = k_2.
\eeq
By our assumption that not all points in $\Set{x_1, x_2 \ldots, x_J}$ are rationally equivalent it
follows that \eqref{sum-3} is a proper subtiling of \eqref{sum}, which forces $k_1=1$.
Adding and subtracting \eqref{sum-3} and \eqref{diff-3} and dividing by 2 we obtain the tilings
\beql{tilings-3}
\frac{1+k_2}{2} = a^{'+}*\omega^+(x) + a^{'-}*\omega^-(x),\ \ \ \frac{1-k_2}{2} = a^{'-}*\omega^+(x) + a^{'+}*\omega^-(x),
\eeq
analogous to \eqref{tilings} and proper subtilings of those in \eqref{tilings}.
But the tilings in \eqref{tilings} are minimal, as they are tilings by translates of sets, at level 1, a contradiction.

In case \eqref{symmetric-case} of the Theorem,
if $\omega^+ \equiv \omega^-$ and $\omega^+$ tiles $\RR$ with some set $a$, then \eqref{diff} is trivial and \eqref{sum} is clearly valid.
From \eqref{sum} it follows using the same tiling structure theorems that $a$ is a periodic multiset,
and \eqref{tilings} can be satisfied by arbitrarily breaking up the multiset $a$ into two sets (not multisets)
$a^+$ and $a^-$. Nothing more can be said about the sets $a^+$ and $a^-$ in this case.

\end{proof}

%%%%%%%%%%%%%%%%%%%%%%%%%%%%%%%%%%%%%%%%%%%%%%%%%%%%%%%%%%%%%%%%%%%%%%%%%%%%%%%%%
%%%%%%%%%%%%%%%%%%%%%%%%%%%%%%%%%%%%%%%%%%%%%%%%%%%%%%%%%%%%%%%%%%%%%%%%%%%%%%%%%
%%%%%%%%%%%%%%%%%%%%%%%%%%%%%%%%%%%%%%%%%%%%%%%%%%%%%%%%%%%%%%%%%%%%%%%%%%%%%%%%%
\section{The structure of a multiplicative tile}
\label{sec:tile-structure}

\subsection{Symmetric tile}\label{sec:symmetric-tile}
Suppose that the set $\Omega$ is symmetric with respect to $0$. In the notation of Theorem \ref{th:mult-structure}
this means that $\Omega^+ = \Omega^-$ (up to measure 0).
In this case the tilings of \eqref{log-tilings} become just one tiling:
\beql{single-tiling}
\RR = (a^+ \cup a^-) + \omega^+.
\eeq
(Remember that the lower case letters denote the logarithms of the sets written in the corresponding upper case.)
So in this case the problem of multiplicative tiling becomes exactly the problem of translational tiling of the real line
by the tile $\omega^+$ and with set of translates the multiset $a^+ \cup a^-$.
The structure of $\omega^+$ in this case has been completely characterized in \cites[Theorem 3]{lagarias1996tiling}.

\subsection{Non-symmetric tile}\label{sec:non-symmetric-tile}
Suppose now that we are in the case where $\omega^+ \nequiv \omega^-$ and, according to Theorem \ref{th:mult-structure},
$a^+, a^-$ are both periodic with the same period obeying \eqref{periodic-rational}.
We can thus write (after scaling)
\beql{a-shape}
a^+ = \ZZ+\frac{1}{L}\Set{0=\alpha_1^+, \alpha_2^+, \ldots , \alpha_m^+},\ \ \ 
a^- = \ZZ+\frac{1}{L}\Set{\alpha_1^-, \alpha_2^-, \ldots , \alpha_n^+},
\eeq
for some positive integer $L$, where $\alpha_j^\pm \in \ZZ_L = \ZZ / (L\ZZ)$.

In order to express our tiling problem on the torus $\TT = \RR / \ZZ$ we identify
$\omega^\pm$ with the indicator function of the set that arises when taking their
projection mod 1.
(Because of the tiling assumption the points of each $\omega^\pm$ are different mod 1.)

Write
$$
\alpha^+ = \Set{\alpha_1^+, \alpha_2^+, \ldots, \alpha_m^+},\ \ \ \alpha^- = \Set{\alpha_1^-, \alpha_2^-, \ldots, \alpha_n^-},
$$
and also $\alpha^\pm$ for the collection of unit Dirac masses at the points of $\alpha^\pm$.

The tiling conditions \eqref{tilings} now become equivalently the tilings of the torus
\beql{mod-tilings}
1 = \frac{\alpha^+}{L}*\omega^+(t) + \frac{\alpha^-}{L}*\omega^-(t) = \frac{\alpha^-}{L}*\omega^+(t) + \frac{\alpha^+}{L}*\omega^-(t),\ \ \ \text{for $t \in \TT$.}
\eeq
Write for $x \in [0,1)$
$$
C_x = \Set{x+\frac{k}{L} \bmod 1:\ k=0,1,\ldots,L-1} \subseteq \TT.
$$
Multiplying \eqref{mod-tilings} by $\one_{C_x}(t)$ we get
$$
\one_{C_x}(t) =
 \frac{\alpha^+}{L}*\one_{C_x}\omega^+(t) + \frac{\alpha^-}{L}*\one_{C_x}\omega^-(t) =
 \frac{\alpha^-}{L}*\one_{C_x}\omega^+(t) + \frac{\alpha^+}{L}*\one_{C_x}\omega^-(t),\ \ \ \text{for $t \in \TT$.}
$$
Restricting to $t \in C_x$ we can rewrite this as the two tiling conditions on $\ZZ_L$, valid for all $x \in [0, \frac{1}{L})$,
\begin{align}
\ZZ_L &= (\omega^+\cap C_x) + \alpha^+ \ \ \cup\ \  (\omega^-\cap C_x) + \alpha^-,\nonumber\\
\ZZ_L &= (\omega^+\cap C_x) + \alpha^- \ \ \cup\ \  (\omega^-\cap C_x) + \alpha^+\label{cycle-tilings},
\end{align}
where we are now identifying the sets $\omega^\pm\cap C_x$ with the obvious subset of $\ZZ_L$.

The sets $\omega^\pm \cap C_x$ can be chosen independently for all $x \in [0, 1/L)$ as tiling is not affected by
what happens on different cosets of $\frac{1}{L}\ZZ$.
We conclude that the sets $\omega^\pm$ are of the form
\begin{align}
\omega^+ &= \bigcup_{x \in [0,1/L)} \left(x + \frac{1}{L} b^+_x \right)\nonumber\\
\omega^- &= \bigcup_{x \in [0,1/L)} \left(x + \frac{1}{L} b^-_x \right) \ \ \ \mbox{(disjoint unions)} \label{decomposition}
\end{align}
where $b^\pm_x \subseteq \ZZ_L$ are such that for each $x$ we have the two tilings
\begin{align}
\ZZ_L &= b^+_x + \alpha^+ \cup b^-_x + \alpha^-\nonumber\\
\ZZ_L &= b^+_x + \alpha^- \cup b^-_x + \alpha^+\label{cycle-tilings-final}
\end{align}
In the next section we are trying to understand better the kind of tiling described in \eqref{cycle-tilings-final}.

%%%%%%%%%%%%%%%%%%%%%%%%%%%%%%%%%%%%%%%%%%%%%%%%%%%%%%%%%%%%%%%%%%%%%%%%%%%%%%%%%
%%%%%%%%%%%%%%%%%%%%%%%%%%%%%%%%%%%%%%%%%%%%%%%%%%%%%%%%%%%%%%%%%%%%%%%%%%%%%%%%%
%%%%%%%%%%%%%%%%%%%%%%%%%%%%%%%%%%%%%%%%%%%%%%%%%%%%%%%%%%%%%%%%%%%%%%%%%%%%%%%%%
\section{Cross tiling}
\label{sec:ct}

\begin{definition}[Cross tiling]\label{def:cross-tiling}\ \\
Suppose $N>1$ is a positive integer and $A, B, X, Y \subseteq \ZZ_N$. We say that the pair $A, B$ admits
{\em cross tiling} with complements $X, Y$ if the following tilings hold:
\begin{align}
\ZZ_N = (A + X) \cup (B + Y)\nonumber\\
\ZZ_N = (A + Y) \cup (B + X) \label{cross-tiling}
\end{align}
A cross tiling \eqref{cross-tiling} is called {\em trivial} if
$(A \cup B) + X = \ZZ_N$ is a tiling and $X=Y$, or the same with the roles of
$A, B$ exchanged with those of $X, Y$.
(It is obvious that in this case conditions \eqref{cross-tiling} are indeed satisfied.)
\end{definition}

\noindent{\bf Remark.}
It is interesting to observe that cross tiling is really an ordinary tiling by translation, although of a larger group.
With $A, B, X, Y \subseteq \ZZ_N$ as in Definition \ref{def:cross-tiling} above write
\begin{align*}
\Gamma &= \ZZ_N \times \ZZ_2,\\
C &= A\times\Set{0} \cup B\times\Set{1},\\
Z &= X\times\Set{0} \cup Y\times\Set{1}.
\end{align*}
It is easy to see that the cross tiling condition \eqref{cross-tiling} is equivalent to the tiling by translation
$$
\Gamma = C + Z.
$$
This alternative characterization of cross tiling may be of interest but is not exploited in this paper

The following two examples are non-trivial cross tilings.
\begin{example}\label{ex:first-non-trivial}
Let $N=2ab$, with odd $a, b \in \NN$ and view $G=\ZZ_N$ as the cross-product $\ZZ_N = \ZZ_{ab}\times\ZZ_2$.
Define
\begin{align*}
A &= \Set{0, 1, 2, \ldots, a-1} \times \Set{0},\\
B &= \Set{0, 1} \cup \Set{a+2, a+3,  \ldots, 2a-1} \times \Set{0},\\
X &= \Set{0, a, 2a, \ldots, (b-1)a} \times \Set{0},\\
Y &= \Set{0, a, 2a, \ldots, (b-1)a} \times \Set{1}.
\end{align*}
It follows that
$$
A + X = B + X = \ZZ_{ab}\times\Set{0}\ \ \ \text{and}\ \ \ B + Y = A + Y =  \ZZ_{ab}\times\Set{1}
$$
so \eqref{cross-tiling} follows. See Figure \ref{fig:easy-example}.

\end{example}

\begin{figure}[h]
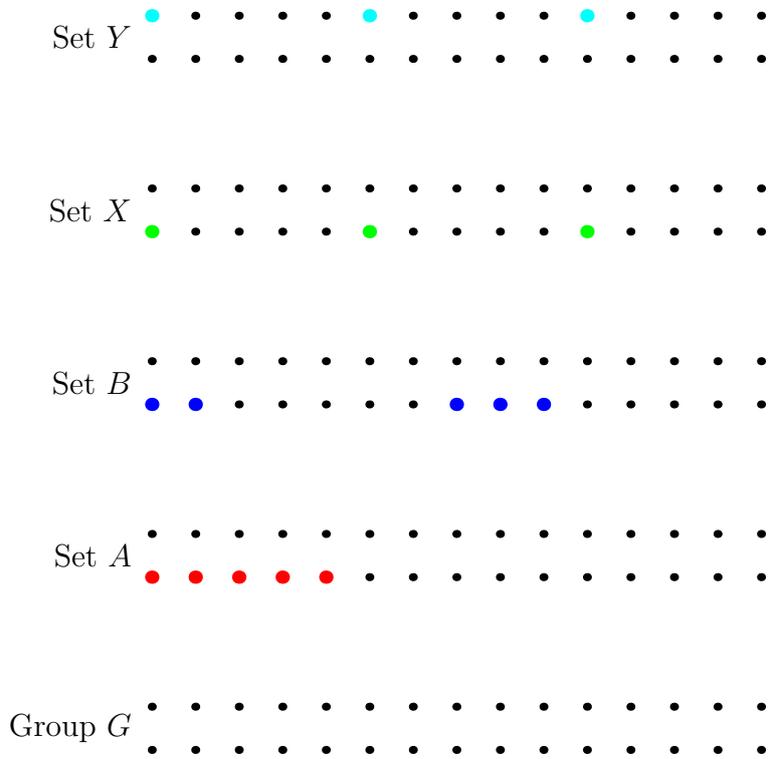

\begin{center}
\begin{asy}
size(10cm,0);
 
int i, j, a=5, b=3, N=2*a*b;
 
pair A[]; for(i=0; i<a; ++i) { A[i] = (i,0); }

pair B[]; B[0] = (0,0); B[1] = (1,0); for(j=2, i=a+2; i<2*a; ++i, ++j) { B[j] = (i,0); }

pair X[]; for(i=0; i<b; ++i) { X[i] = (i*a, 0); }

pair Y[]; for(i=0; i<b; ++i) { Y[i] = (i*a, 1); }
 
pair u=(0,4);

// Draw group
picture pg;
for(i=0; i<a*b; ++i) {
 for(j=0; j<2; ++j) {
  draw(pg, (i,j), 3bp+black);
 }
}
add(pg);
label("Group $G$ ", (0,.5), W);

// Sets
picture pa;
for(i=0; i<A.length; ++i) {
 draw(pa, A[i], 5bp+red);
}
add(shift(u)*pg);
add(shift(u)*pa);
label("Set $A$ ", u+(0,.5), W);

picture pb;
for(i=0; i<B.length; ++i) {
 draw(pb, B[i], 5bp+blue);
}
add(shift(2*u)*pg);
add(shift(2*u)*pb);
label("Set $B$ ", 2*u+(0,.5), W);

picture px;
for(i=0; i<X.length; ++i) {
 draw(px, X[i], 5bp+green);
}
add(shift(3*u)*pg);
add(shift(3*u)*px);
label("Set $X$ ", 3*u+(0,.5), W);

picture py;
for(i=0; i<Y.length; ++i) {
 draw(py, Y[i], 5bp+cyan);
}
add(shift(4*u)*pg);
add(shift(4*u)*py);
label("Set $Y$ ", 4*u+(0,.5), W);

\end{asy}
\end{center}
\caption{The sets cross tiling in Example \ref{ex:first-non-trivial} with $a=5, b=3$} \label{fig:easy-example}
\end{figure}

\begin{figure}[h]
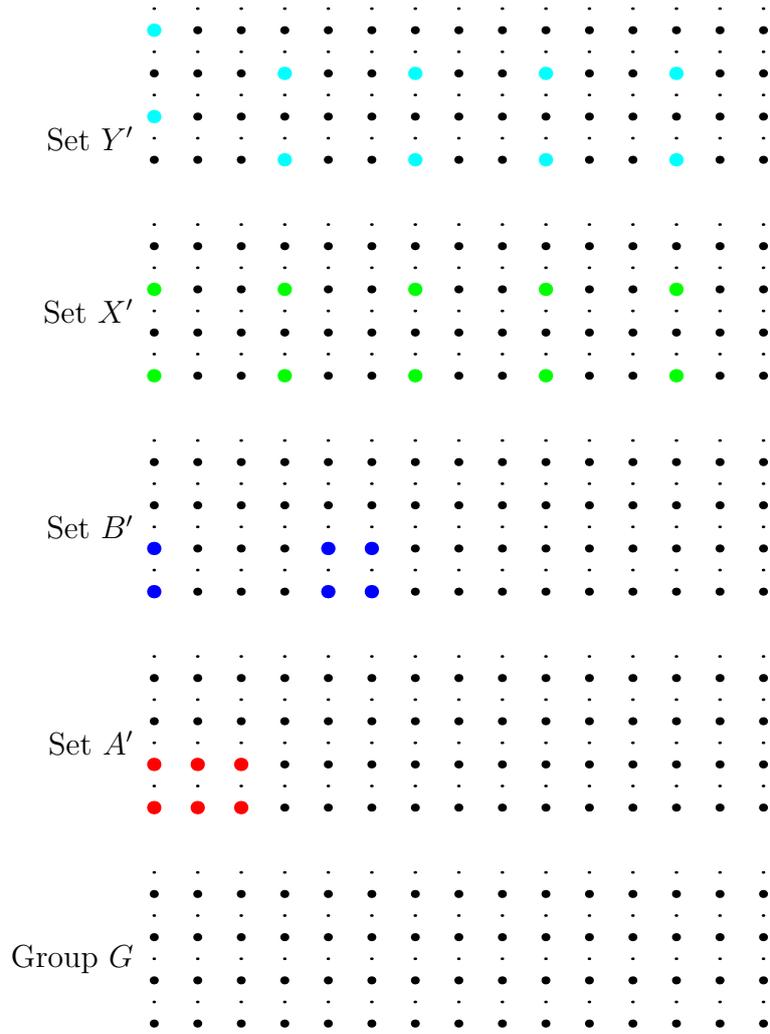

\begin{center}
\begin{asy}
size(10cm,0);

int i, j, a=15, b=4;

real F1[]={0,1,2}, F2[]={0,4,5}, C[]={0, 3, 6, 9, 12}, CC[]={3, 6, 9, 12};

pair A[]; for(i=0; i<F1.length; ++i) {A.push((F1[i], 0)); A.push((F1[i], 1)); }

pair B[]; for(i=0; i<F2.length; ++i) {B.push((F2[i], 0)); B.push((F2[i], 1)); }

pair X[]; for(i=0; i<C.length; ++i) {X.push((C[i], 0)); X.push((C[i], 2)); }

pair Y[]; Y.push((0,1)); Y.push((0,3)); for(i=0; i<CC.length; ++i) {Y.push((CC[i], 0)); Y.push((CC[i], 2)); }

pair u=(0,5);

// Draw group
picture pg;
for(i=0; i<a; ++i) {
for(j=0; j<b; ++j) {
draw(pg, (i,j), 3bp+black);
}
}
for(i=0; i<a; ++i) {
for(j=0; j<b; ++j) {
//draw(pg, (i+0.5,j), 1bp+black);
//draw(pg, (i+0.5,j+0.5), 1bp+black);
draw(pg, (i,j+0.5), 1bp+black);
}
}
add(pg);
label("Group $G$ ", (0,1.5), W);

// Sets
picture pa;
for(i=0; i<A.length; ++i) {
draw(pa, A[i], 5bp+red);
}
add(shift(u)*pg);
add(shift(u)*pa);
label("Set $A'$ ", u+(0,1.5), W);

picture pb;
for(i=0; i<B.length; ++i) {
draw(pb, B[i], 5bp+blue);
}
add(shift(2*u)*pg);
add(shift(2*u)*pb);
label("Set $B'$ ", 2*u+(0,1.5), W);

picture px;
for(i=0; i<X.length; ++i) {
draw(px, X[i], 5bp+green);
}
add(shift(3*u)*pg);
add(shift(3*u)*px);
label("Set $X'$ ", 3*u+(0,1.5), W);

picture py;
for(i=0; i<Y.length; ++i) {
draw(py, Y[i], 5bp+cyan);
}
add(shift(4*u)*pg);
add(shift(4*u)*py);
label("Set $Y'$ ", 4*u+(0,.5), W);

\end{asy}
\end{center}
\caption{The sets $A', B', X', Y' \subseteq H$ in Example \ref{ex:second-non-trivial} giving rise to the sets $A, B, X, Y \subseteq G$} \label{fig:better-example}
\end{figure}

\begin{example}\label{ex:second-non-trivial}
In Example \ref{ex:first-non-trivial} the sets $X$ and $Y$ are translates of each other, which is perhaps not very
satisfactory in terms of deviating from triviality.  
Now we give an example where this does not happen.

We work in the group $G = \ZZ_{15}\times\ZZ_8 = \ZZ_{120}$
which contains the subgroup
$$
H = \ZZ_{15} \times \Set{0, 2, 4, 6}.
$$

Next define the subsets of $\ZZ_{15}$
$$
F_1 = \Set{0, 1, 2},\ \ \ F_2 = \Set{0, 4, 5}
$$
and notice that they both tile $\ZZ_{15}$ with the complement $\Set{0, 3, 6, 9, 12}$.

Define the subsets of $H$
\begin{align*}
A' &= F_1 \times \Set{0, 2},\\
B' &= F_2 \times \Set{0, 2},\\
X' &= \Set{0, 3, 6, 9, 12} \times \Set{0, 4},\\
Y' &= (\Set{0}\times\Set{1, 3}) \cup (\Set{3, 6, 9, 12}) \times \Set{0, 4}.
\end{align*}
See Figure \ref{fig:better-example} where the group $H$ is shown as thick dots
while its single coset in $G$ is shown as thin dots.
The difference between $X'$ and $Y'$ is that the first ``column'' of $Y'$ is ``raised'' by 1.

It is easy to verify that each of $A', B'$ tiles $H$ with each of $X', Y'$ as a tiling complement.

Finally define the subsets of $G$
\begin{align*}
A &= A', \\
B &= B', \\
X &= X', \\
Y &= Y'+(0,1),
\end{align*}
and observe that they do satisfy the cross-tiling conditions
\begin{align}
G &= (A+X) \cup (B+Y) \nonumber \\
G &= (A+Y) \cup (B+X). \label{ct}
\end{align}
The sets in parentheses in \eqref{ct} are tilings of each of the two $H$-cosets in $G$.
The sets with $+X$ tile $H$ and those with $+Y$ tile $H+(0,1)$.

None of the sets $A, B, X, Y$ is a translate of another.
\end{example}

%%%%%%%%%%%%%%%%%%%%%%%%%%%%%%%%%%%%%%%%%%%%%%%%%%%%%%%%%%%%%%%%%%%%%%%%%%%%%%%%%
%%%%%%%%%%%%%%%%%%%%%%%%%%%%%%%%%%%%%%%%%%%%%%%%%%%%%%%%%%%%%%%%%%%%%%%%%%%%%%%%%
%%%%%%%%%%%%%%%%%%%%%%%%%%%%%%%%%%%%%%%%%%%%%%%%%%%%%%%%%%%%%%%%%%%%%%%%%%%%%%%%%
\subsection{Fourier condition for cross tiling}

Translational tiling $A + X = \ZZ_N$ has a simple equivalent Fourier condition (we use the same letter for a set and its indicator function):
\beql{tiling-ft}
\ft{A}(0) \ft{X}(0) = N,\ \ \ \mbox{and}\ \ \ \forall k \in \ZZ_N\setminus\Set{0}: \ft{A}(k) \neq 0 \Longrightarrow \ft{X}(k) = 0.
\eeq
Adding and subtracting the cross-tiling defining conditions
$$
A*X+B*Y \equiv 1,\ \ \ A*Y+B*X \equiv 1,
$$
we obtain the equivalent conditions
\beql{cross-tiling-equiv}
(A+B)*(X+Y) \equiv 2,\ \ \ (A-B)*(X-Y) \equiv 0.
\eeq
Taking Fourier Transforms these conditions lead to the following equivalent Fourier condition for cross-tiling:
\begin{align}
&(\ft{A}(0) + \ft{B}(0)) (\ft{X}(0)+\ft{Y}(0)) = 2N,\nonumber\\
\forall k \in \ZZ_N\setminus\Set{0}:\ \ \ &\ft{A}(k) \neq -\ft{B}(k) \Longrightarrow \ft{X}(k) = -\ft{Y}(k). \label{cross-tiling-ft}\\
\forall k \in \ZZ_N:\ \ \ &\ft{A}(k) \neq \ft{B}(k) \Longrightarrow \ft{X}(k) = \ft{Y}(k).\nonumber
\end{align}
Using $k=0$ in the last set of equations we obtain that necessarily
$$
\Abs{A} = \Abs{B}\ \ \ \text{or}\ \ \ \Abs{X} = \Abs{Y}.
$$

\printbibliography

\end{document}